   \documentclass[twoside,reqno,11pt]{fcaa-var} %

\usepackage{graphicx}
\usepackage{epsfig}

\usepackage{amsthm}
\usepackage{amsmath}
\usepackage{latexsym}
\usepackage{amsfonts}
\usepackage{amssymb}

\newtheoremstyle{theorem}
  {15pt}          
  {15pt}  
  {\sl}  
  {\parindent}
  {\sc}  
  {. }   
  { }    
  {}     
\theoremstyle{theorem}

\newtheorem{theorem}{Theorem}[section]
\newtheorem{corollary}{Corollary}[section]

\newtheoremstyle{defi}
  {15pt}          
  {15pt}  
  {\rm}  
  {\parindent}     
  {\sc}  
  {. }    
  { }    
  {}     
\theoremstyle{defi}
\newtheorem{definition}{Definition}[section]
\newtheorem{remark}{Remark}[section]
\newtheorem{example}{Example}[section]

 \def\proofname{\indent\rm P\ r\ o\ o\ f.}
 
 \def\qed{\hfill$\Box$} 

\usepackage{fge,bm} 
 \usepackage{hyperref} %

 \def\theequation{\arabic{section}.\arabic{equation}}

  \title[On the solution of two-sided fractional \dots]
       {On the solution of two-sided fractional \\ [3pt] 
       ordinary differential equations of \\[3pt]  Caputo type }
 \author[\normalsize M.E. Hern\'andez-Hern\'andez, V.N. Kolokoltsov]
  {\normalsize Ma. Elena Hern\'andez-Hern\'andez $^1$, Vassili N. Kolokoltsov $^2$}

\begin{document}

 \vbox to 2.5cm { \vfill }

 \bigskip \medskip

 \begin{abstract}
 
This paper provides well-posedness results and stochastic representations for the solutions to  equations involving both the right- and the left-sided generalized  operators  of  Caputo type. As a special case, these results   show the interplay between two-sided  fractional differential equations and  two-sided exit problems for certain  L\'evy processes.
 
 \medskip

{\it MSC 2010\/}: Primary 34A08;
                  Secondary 35S15, 26A33, 60H30
 \smallskip

{\it Key Words and Phrases}: two-sided fractional equations, generalized Caputo  type derivatives, boundary point, stopping time, Feller process, L\'evy process

 \end{abstract}

 \maketitle

 \vspace*{-22pt}



 \section{Introduction} \label{sec:1}
\setcounter{section}{1}
\setcounter{equation}{0}\setcounter{theorem}{0}

The successful use of classical fractional derivatives  to describe, for example,    relaxation  phenomena, processes of oscillation,  viscoelastic systems and  diffusions in disordered media (anomalous diffusions) among others,  have promoted  an increasing research on  the field of  fractional differential equations. For an account of historical notes, applications and different methods to solve fractional equations we refer, e.g., to    \cite{Bouchaud1990}-\cite{kai}, \cite{edwards}, \cite{gorenflo98},  \cite{Pal}-\cite{kilbas2},  \cite{KV},
\cite{FMainardi1997}-\cite{podlubny}, \cite{samko}, \cite{zaslavsky}, and references cited therein.

Apart from the different notions of fractional derivatives  found in the literature (e.g., the Caputo, the Riemann-Liouville,  the Grunwald-Letnikov, the Riesz,  the Weyl,  the Marchaud, and the  Miller and Ross fractional derivatives),    numerous generalizations  (mostly from an analytical point of view)  have been proposed by many authors, we refer, e.g., to \cite{Om2012}, \cite{Hilfer2000}-\cite{Kalla1980}, \cite{Kiryakova1994}-\cite{Kiryakova2008},  \cite{MarkTemperedFC} for details.    As for the \textit{generalized fractional operators of Caputo type}  considered in this work, they were introduced in \cite{KVFDE} by one of the authors as  generalizations (from a probabilistic point of view) of the classical    Caputo derivatives of order $\beta \in (0,1)$ when applied to regular enough functions. These Caputo type operators can be thought of as the generators of   Feller  processes \textit{interrupted}  on the first attempt to cross certain boundary point (see precise definition later).

As a continuation of our  previous works, which show a new link between stochastic analysis and fractional equations (see \cite{KV-1}-\cite{KV-2}, \cite{KVFDE}), this paper appeals to a  probabilistic approach  to study   equations  involving both  left-sided   and  right-sided generalized operators of Caputo type. We address   the boundary value problem  for the \textit{two-sided generalized linear equation} with Caputo type derivatives $-D_{a+*}^{(\nu_+)}$ and   $-D_{b-*}^{(\nu_{-})}$:
\begin{align}\nonumber
 -D_{a+*}^{(\nu_+)} u(x)  -D_{b-*}^{(\nu_{-})} u(x) - A u(x)&= \lambda u(x) - g(x),\quad x \in (a,b), \\ u(a)&= u_a,\quad u(b) = u_b,
    \label{Eq0-C}\end{align}
\noindent where $\lambda \ge 0$, $u_a, u_b \in \mathbb{R}$ and  $g$ is a  prescribed function on $[a,b]$. Notation $-A \equiv -A^{(\gamma,\alpha)}$ refers to the second order differential operator
\begin{equation}\label{A}
- A^{(\gamma,\alpha)}: =  \gamma (\cdot) \frac{d}{dx} +  \alpha(\cdot)  \frac{d^2}{dx^2}.
\end{equation}
 Equation (\ref{Eq0-C}) includes, as  special  cases, the fractional equations
\begin{align}\label{Eq0-0}
 D_{a+*}^{\beta_1} u(x)  + D_{b-*}^{\beta_2} u(x) &= g(x),\quad x \in (a,b),\quad \beta_1, \beta_2 \in (0,1), \\u(a)= u_a,\quad  u(b) &= u_b,   \nonumber
  \end{align}
where   $D_{a+*}^{\beta_1}$ and $D_{b-*}^{\beta_2}$ are the left- and the right-sided Caputo derivatives  of order $\beta_1$ and $\beta_2$, respectively. There are  relatively scarce results dealing with two-sided fractional ordinary equations.  For example, to the best of our knowledge, the Riemann-Liouville version of \eqref{Eq0-0}
was analyzed (in the space of distributions)  in \cite{2s2008}-\cite{2s2011}, whereas the  explicit solution to the two-sided fractional   equation  in (\ref{Eq0-0}) was just recently provided    in \cite{KVFDE}.

Another special case of equation (\ref{Eq0-C}) is  the \textit{two-sided} equation:
\begin{align}\label{Eq1-1}
 c_1D_{a+*}^{\beta_1} u(x)  + c_2 D_{b-*}^{\beta_2} u(x) + \gamma(x) u'(x) + \lambda u(x) &= g(x),\quad x \in (a,b), \\u(a)= u_a,\quad \quad  u(b) &= u_b.   \nonumber \end{align}
 If $c_1>0$, $c_2 = 0$, $\beta_1 = \frac{1}{2}$ and $\lambda = 1$, then the  (one-sided) equation is known as the  \textit{Basset equation}, well-studied in the literature (see, e.g. \cite{FMainardi1997} and references therein).
The one-sided  case with $\beta_1 \in (0,1)$ (known as the \textit{composite fractional relaxation equation}) was treated via the Laplace transform method in \cite[Section 4]{gorenflo2008}, whereas the left-sided case  with Caputo type and RL type operators  was studied by the authors in \cite{KV-2}.

Some other examples   showing the relevance of left- and right-sided derivatives in mathematical modeling  appear in the study of FPDE's on bounded domains, as well as in fractional calculus of variations, see, e.g.,\cite{Agrawal2002},
  \cite{teodor},  \cite{Pal},  \cite{Mark2sided},        \cite{CTorres}.

In this paper we study the well-posedness of \eqref{Eq0-C}  by considering two types of solutions: \textit{solutions in the domain of the generator}  and \textit{generalized solutions}. The first type is understood as a solution $u$ that belongs to the domain of the two-sided operator  seen as the generator of a Feller process.  Since the existence of such a solution is quite restrictive once one imposes boundary conditions, the notion of  generalized solution is introduced via the  limit of approximating   solutions taken from the domain of the generator.

Further, appealing to the relationship between two-sided equations and  exit problems for Feller processes (already mentioned in \cite{KVFDE}),  we provide some  explicit solutions to   two-sided equations in the context of classical fractional derivatives.
 Even though   exit problems for L\'evy processes have been widely studied (see, e.g., \cite{bertoin96}-\cite{Getoor},  \cite{Watson}, \cite{watanabe}),    to our knowledge fractional equations of the type in (\ref{Eq0-0}) and their connection with exit problems  seem to be novel in the literature.  We believe that the probabilistic solutions presented in this work can be used, for example, to obtain  numerical  solutions to  classical fractional equations for which explicit solutions are unknown.

\smallskip

 The paper is organized as follows. The next Section \ref{sec:2} sets  standard notation and definitions.  
 Section \ref{sec:3} gives a quick review about generalized Caputo type operators.      
 Section \ref{sec:4} provides  preliminary results concerning   \textit{two-sided generalized  operators} and their connections 
 with the generators of Feller processes.  Then, Section \ref{sec:5} addresses the  well-posedness  for  the RL type version 
 of (\ref{Eq0-C}).  The study of the  Caputo type equation (\ref{Eq0-C})  is given in Section \ref{sec:6}.    
 Some examples  are presented in Section \ref{sec:7}.  
 Finally, Section \ref{sec:8} contains the proofs of  some key  results established in  Section \ref{sec:4}.
 
 \vspace*{-4pt}
 
 \section{Preliminaries} \label{sec:2} 
 \setcounter{section}{2}
\setcounter{equation}{0}\setcounter{theorem}{0}

 \subsection{Notation}\label{section2.1}
Let $\mathbb{N}$  and $\mathbb{R}$  be the set of positive integers and  the real line, respectively.  For any open set $A \subset \mathbb{R}$,  notation  $B(A)$, $C(A)$ and    $C_{\infty}(A)$   denote the set of bounded Borel measurable functions, bounded continuous functions and   continuous functions vanishing at infinity defined on $A$, respectively, equipped with the sup-norm $||h|| = \sup_{x \in A} |h(x)|$.   The space of continuous functions  on  $A$  with continuous derivatives up to and including order $k$  is denoted by  $C^{k}(A)$. This space is equipped with the norm $|| h ||_{C^k} := ||h|| + \sum_{k=1} ||h^{(k)}||$.
  For  functions defined on the closure $\bar{A}$ of $A$, notation  $C^k(\bar{A})$ means the space of $k$ times continuously  differentiable functions up to the boundary.
Further,   spaces $C_0[a,b]$ and $C_0^k[a,b]$  stand for the space of continuos functions on $[a,b]$   vanishing at the boundary and   the space of functions $C_0[a,b]\cap C^k[a,b]$, respectively.

Letters $\mathbf{P}$ and $\mathbf{E}$  are reserved for the probability and the mathematical expectation, respectively.  For a stochastic process $X_x=(X_x(t))_{t\ge 0}$ with state space $A$, the subscript $x$ in  $X_x(t)$ means that the process starts at $x \in A$, so that notation $ \mathbf{E} \left[  f \left( X_x (t) \right ) \right]$ is  understood as  $\mathbf{E} \left[  f \left( X (t) \right )  | X(0) = x\right]$.  All the processes considered in this paper are assumed to be defined on a fixed complete probability space $(\Omega, \mathcal{F},  \mathbb{P})$.
\subsection{Feller processes: basic definitions}
Let  $\{T_t\}_{t \ge 0}$ be a strongly continuous semigroup  of linear bounded operators on a Banach space $(B, ||\cdot||_B)$, i.e. $\lim_{t \to 0} || T_t f - f ||_{B} = 0$ for all $f \in B$.  Its (infinitesimal) generator $L$ with domain $\mathfrak{D}_L$, shortly $(L, \mathfrak{D}_L)$,  is defined  as the (possibly unbounded) operator $L : \mathfrak{D}_L \subset B \to B$ given by the strong limit
\begin{equation}\label{Gdef}
L  f  := \lim_{t \downarrow 0} \frac{T_t f- f}{ t}, \quad f\in \mathfrak{D}_L,
\end{equation}
where the domain of the generator $\mathfrak{D}_L$   consists of those   $f \in B$ for which the limit in  (\ref{Gdef}) exists in the norm sense.  We also recall that, if $L$ is a closed operator, then a linear subspace $\mathcal{C}_L \subset \mathfrak{D}_L$ is called a \textit{core} for the generator $L$ if the operator $L$ is the closure of the restriction $L\big |_{\mathcal{C}_L}$ \cite[Chapter 1, Section 3]{EK}.  If  additionally $T_t \mathcal{C}_L  \subset \mathcal{C}_L$ for all $t\ge 0$, then $\mathcal{C}_L$ is said to be an  \textit{invariant core}.
The  \textit{resolvent operator} $R_{\lambda}$ of the semigroup $\{T_t\}_{t\ge 0}$  is defined (for any $\lambda > 0$) as the Bochner integral (see,  e.g.,  \cite[Chapter 1]{dynkin1965}, \cite[Chapter 1]{EK})
\begin{equation}\label{resolvent}
R_{\lambda} g\,:= \int_0^{\infty} e^{-\lambda t} T_t g\,dt, \quad g \in B.
\end{equation}
By taking $\lambda = 0$ in (\ref{resolvent}), one  obtains  the \textit{potential operator}  denoted by $R_0g$ (whenever it exists).

We say that a (time-homogeneous) Markov process $X=(X(t))_{t\ge 0}$ taking values on $A \subset \mathbb{R}^d$   is a \textit{Feller process} (see, e.g., \cite[Section 3.6]{KV})  if its semigroup $\{T_t\}_{t\ge0}$,
defined  by
\[  T_t f (x)  := \mathbf{E} \left[  f \left( X (t) \right )  | X(0) = x\right], \quad t \ge 0, \,\,\, x \in A, \,\,\,  f \in B(A),\]
 gives rise to a \textit{Feller semigroup}  when reduced to $C_{\infty} (A)$, i.e. it is a strongly continuous semigroup  on $C_{\infty} (A)$ and it is formed  by  positive linear contractions  ($0 \le T_t f \le 1$ whenever $0 \le  f \le 1$).


\section{Generalized fractional operators of Caputo type and RL type} \label{sec:3}
\setcounter{section}{3}
\setcounter{equation}{0}\setcounter{theorem}{0}

The generalized Caputo type operators introduced in \cite{KVFDE} are defined in terms of a function $\nu: \mathbb{R} \times \left( \mathbb{R} \setminus
\{0\}\right) \to \mathbb{R}^+$ satisfying the condition:
 \begin{enumerate}
  \item [\textbf{(H0)}] The function $\nu(x,y)$ is  continuous as a function of two variables and continuously  differentiable  in the first variable.    Furthermore,
  \begin{equation}\nonumber
    \sup_x \int \min\{1,|y|\} \nu (x,y)dy < \infty, \quad \sup_x \int \min \{1,|y|\} \Big | \frac{\partial}{\partial x} \nu(x,y)\Big |dy < \infty,
  \end{equation}
  and
  \begin{equation}\nonumber
    \lim_{\delta \to 0} \sup_x \int_{|y| \le \delta} |y| \nu(x,y)dy = 0.
   \end{equation}
  \end{enumerate}

 \begin{definition}   
 Let $a,b \in \mathbb{R}$ with $a < b$. For any function $\nu$ satisfying the condition (H0), the operators $-D_{a+*}^{(\nu)}$ and $-D_{b-*}^{(\nu)}$ defined  by
  \begin{align}
 \left( -D_{a+*}^{(\nu)} h\right)(x)      = \int_0^{x-a}  ( h(x-y) &- h(x)) \nu(x,y)dy \,\,+ \nonumber \\ &+ (h(a) - h(x)) \int_{x-a}^{\infty}\nu(x,y)dy, \label{genaP}
  \end{align}
  for functions $h: [a, \infty) \to \mathbb{R}$, and by
   \begin{align}
\left(-D_{b-*}^{(\nu)} h\right) (x)   = \int_0^{b-x} ( h(x+y) &- h(x)) \nu(x,y)dy \,\,+ \nonumber \\ &+ (h(b) - h(x)) \int_{b-x}^{\infty}\nu(x,y)dy,  \label{genaP-}
  \end{align}
  for functions $h: (-\infty,b] \to \mathbb{R}$,
are  called \textit{the generalized left-sided Caputo type operator} and  \textit{the generalized right-sided Caputo type operator}, respectively.   The values $a$ and $b$ will be referred to as the \textit{terminals} of the corresponding  operators.
\end{definition}

\begin{remark} 
The sign $-$ appearing in the previous notation  is introduced to comply with the standard notation of fractional derivatives.
\end{remark}

Due to assumption (H0),  the  operators (\ref{genaP})-(\ref{genaP-}) are well defined at least on the space of continuously differentiable functions (with bounded derivative).

\begin{remark}\label{RLDefinition}
The \textit{left-sided} (resp. \textit{right-sided}) \textit{generalized Riemann-Liouville type operator} $-D_{a+}^{(\nu)}$ (resp. $-D_{b-}^{(\nu)}$) is defined by setting $h(a) = 0$ (resp.  $h(b)=0$) in \eqref{genaP} (resp. \eqref{genaP-}). Hence,
\begin{equation}\nonumber-D_{a+*}^{(\nu)} h(x) =- D_{a+}^{(\nu)} [ h - h(a)](x) \quad \text{ and } \quad-D_{b-*}^{(\nu)} h(x) = - D_{b-}^{(\nu)} [ h - h(b)](x).
\end{equation}
\end{remark}

 \subsubsection{Particular cases.}
For smooth enough functions $h$,  the standard analytical  definitions of the left-sided Caputo derivative $D_{a+*}^{\beta}$ and the right-sided Caputo derivatives $D_{b-*}^{\beta}$ of order $\beta \in (0,1)$ (see, e.g., \cite[Definition 2.2, Definition 3.1]{kai}) can be rewritten as (see, e.g., \cite[Appendix]{KVFDE})
\begin{align}\label{Caputo}
  \left(D_{a+*}^{\beta}h\right) (x) &= \frac{\beta}{\Gamma ( 1- \beta)} \int_0^{x-a} \frac{h(x-y) - h(x) }{y^{1+\beta}} dy - \frac{h(x) - h(a)}{ \Gamma (1-\beta) (x-a)^{\beta}},
  \end{align}
and
\begin{align}\label{CaputoRight}
 \left( D_{b-*}^{\beta}h \right)(x) &= \frac{\beta}{\Gamma (1- \beta)} \int_0^{b-x} \frac{h(x+y) - h(x) }{y^{1+\beta}} dy - \frac{h(x) - h(b)}{ \Gamma (1-\beta) (b-x)^{\beta}}.
  \end{align}
Hence, for $h$ regular enough, $D_{a+*}^{\beta} h$ (resp. $-D_{b-*}^{\beta}h$) is a particular case of $ -D_{a+*}^{(\nu)}h$ (resp.  $-D_{b-*}^{\beta} h$) obtained by taking  the function  \begin{equation}\label{2.7a}
 \nu(x,y)\equiv \nu(y) = \, -\, \frac{\beta}{\Gamma(1-\beta) y^{1+\beta}} ,   \quad \beta \in (0,1).
\end{equation}

  \begin{remark}
 Other examples of generalized  operators  $-D_{a+*}^{(\nu)}$ include the \emph{fractional derivatives of variable order}, as well as the \emph{generalized distributed order fractional derivatives} (see \cite{KV-1}, \cite{KVFDE} for precise definitions).
 \end{remark}

\section{Two-sided operators of RL type and  Caputo type} \label{sec:4}
\setcounter{section}{4}
\setcounter{equation}{0}\setcounter{theorem}{0}

Given two functions  $\nu_+$ and $\nu_{-}$  satisfying  condition (H0),  define the   function $\nu \,: \mathbb{R} \times \mathbb{R}\setminus\{0\} \to \mathbb{R}^+$ associated with $\nu_+$ and $\nu_{-}$ by setting
\begin{equation}\label{nu+-}
 \nu(x,y) := \nu_+(x,y),\quad y > 0,\quad\quad \nu(x,y) := \nu_{-}(x,-y), \quad y <0.
 \end{equation}
Define the \textit{two-sided operator of RL type} $-L_{[a,b]}$ and  the \textit{two-sided operator of Caputo type}  $-L_{[a,b]*}$  by
\begin{align}
\left(- L_{[a,b]}f \right)(x)&:=
\left ( -D_{a+}^{(\nu_+)} f \right )(x) + \left( -D_{b-}^{(\nu_{-})} f \right ) (x) + \left (- A^{(\gamma,\alpha)}  f  \right  ) (x),
\label{DabA}  \end{align}
and
\begin{align}
\left(- L_{[a,b]*} f \right)(x)&:= \left ( -D_{a+*}^{(\nu_+)} f \right )(x) + \left( -D_{b-*}^{(\nu_{-})} f \right ) (x) + \left (- A^{(\gamma,\alpha)}  f  \right  ) (x).\label{DabA*}  \end{align}
 Notation $-A^{(\gamma,\alpha)}$ stands for the differential operator given  in (\ref{A}).  We will see that the operator $-L_{[a,b]*}$  can be thought of as the generator of a Feller process on $[a,b]$, whereas $-L_{[a,b]}$ is related to the generator of a killed process. For that purpose,  let us  introduce an additional definition     for the regularity of the boundary (see, e.g., \cite[Chapter 6]{KV0}).

  \begin{definition} 
  For a domain $D\subset \mathbb{R}$ with boundary $\partial D$, a point $x_0 \in \partial D$ is said to be  \textit{regular in expectation} for a Markov process $X$ (or for its generator ) if $\,\,\mathbf{E} \left [ \tau_{D}(x) \right ] \to 0,$ as $x \to x_0,\,\, x \in D,$
   where
   $\tau_{D} (x):= \inf \left \{ t \ge 0\,:\, X_x(t) \notin D \right \}$, with the usual convention that $\inf \{ \varnothing\} = \infty$.
 \end{definition}

\begin{theorem}  \label{L0diffusion-int}
 Let $\nu$   be a function satisfying  assumption  (H0).  Suppose  that   $\gamma \in C_0^3[a,b]$, $\alpha  \in C^3[a,b]$ with derivative $\alpha' \in C_0[a,b]$ and  $ \alpha$ being a positive function.  Then, 
 \smallskip
 
 $(i)$ the  operator $(\,-L_{[a,b]*},\,\hat{\mathfrak{D}}_{*} \,)$
generates a Feller process $\hat{X}$  on $[a,b]$  with a domain  $\hat{\mathfrak{D}}_{*}$ such that
\begin{equation}
\left \{  f \in C^2[a,b]\,:\, f' \in C_0[a,b] \right \}\subset \hat{\mathfrak{D}}_{*}.
\end{equation}
  
  $(ii)$  The points $\{a,b\}$  are regular in expectation for $(-L_{[a,b]*}, \hat{\mathfrak{D}}_{*})$.   Further, the first exit time $\hat{\tau}_{(a,b)}(x)$ from the interval $(a,b)$ of  $\hat{X}_x$,  $x \in (a,b)$, has a finite  expectation.
 \end{theorem}

  \begin{proof}
 See proof in Section \ref{sec:8}.
 \end{proof}

\textit{Stopped and killed processes.} 
To introduce the notion of solutions to the equation  (\ref{Eq0-C}) we are interested in,  we  need the stopped version of  $\hat{X}$.

 \begin{theorem}\label{L0diffusion}
Suppose that the assumptions of Theorem \ref{L0diffusion-int} hold.  Let $\hat{X}_x$ be the process  started at $x \in (a,b)$ generated by $(\,-L_{[a,b]*},\,  \hat{\mathfrak{D}}_{*}\,)$.
\begin{itemize}
\item [(i)]The process $X_x^{[a,b]*}$ defined by $X_x^{[a,b]*}(s) :=  \hat{X}_x (s \wedge \hat{\tau}_{(a,b)}(x))$,  $s \ge 0$, is a Feller process on $[a,b]$. If the operator  $(-L_{stop}, \mathfrak{D}_{[a,b]*}^{stop})$ denotes the generator of $X^{[a,b]*}$,    then  for any $f \in \hat{\mathfrak{D}}_{*}$ satisfying $\left(-L_{[a,b]*} f\right) (x) = 0$ for $x \in \{a,b\}$, it holds that  $f \in \mathfrak{D}_{[a,b]*}^{stop}$ and $-L_{stop} f = -L_{[a,b]*} f$.
\item [(ii)]The process $X_x^{[a,b]}$ defined by $X_x^{[a,b]}(s) :=  X_x^{[a,b]*} (s)$ for $s < \hat{\tau}_{(a,b)}(x)$ is a Feller (sub-Markov) process on $(a,b)$. If
 $(-L_{kill}, \mathfrak{D}_{[a,b]}^{kill})$ denotes the generator of  $X^{[a,b]}$,   then for any  $f \in \mathfrak{D}_{[a,b]*}^{stop}$ satisfying $f (x) = 0$ for $x \in \{a,b\}$, it holds that  $f \in \mathfrak{D}_{[a,b]}^{kill}$ and $-L_{kill} f = -L_{[a,b]}f $.
\end{itemize}
\end{theorem}

\begin{proof}
See proof in Section \ref{sec:8}.
\end{proof}

\begin{remark}\label{paths}  
The operator $-L_{[a,b]*}$ can  be obtained from the generator $(L,\mathfrak{D}_L)$ of a  Feller process, say $X_x$,  given by\begin{align}\label{Lfree}
(L f )(x) = \int_{-\infty}^{\infty} (\,f(x+y) - f(x) \,) \nu(x,y)dy + \gamma (x) f'(x) + \alpha (x)f''(x),
\end{align}
  by modifying it in such a way that it forces  the   jumps aimed to be out of the interval $(a,b)$ to  land at the nearest (boundary) point (see also \cite{KVFDE}). If, instead, the process is killed upon leaving $(a,b)$, then the corresponding process has a generator related to the operator $-L_{[a,b]}$.
Thus, when starting at the same state $x \in (a,b)$, it holds that the paths of the processes $X_x$, $\hat{X}_x$, $X_x^{[a,b]*}$ and $X_x^{[a,b]}$ coincide before their first exit time from the interval $(a,b)$. Hence, the first exit time in all cases will always be denoted by $\tau_{(a,b)}(x)$. We refer to the processes $X_x$, $\hat{X}_x$, $X^{[a,b]*}_x$ and $X_x^{[a,b]}$ as the \textit{underlying} process, the \textit{interrupted} process, the \textit{stopped} process and the  \textit{killed} process, respectively.
\end{remark}

\section{Two-sided equations involving RL type operators} \label{sec:5}
\setcounter{section}{5}
\setcounter{equation}{0}\setcounter{theorem}{0}

Let us now study the equation (\ref{Eq0-C}) for which we will also use the short notation $\left(-L_{[a,b]*}, \lambda, g, u_a, u_b\right)$. We shall  start with the boundary value problem with zero boundary conditions: $u_a = 0 = u_b$. Thus, due to the relationship between Caputo and RL type operators (see Remark \ref{RLDefinition}),  the two-sided Caputo type operator $-L_{[a,b]*}$ can be replaced with the  RL type operator $-L_{[a,b]}$, so that the equation $\left(-L_{[a,b]}, \lambda, g, 0,0\right)$ will be called the two-sided RL type equation.

\begin{definition}(\textbf{Solutions to RL type equations})\label{Def3Sol}
 Let $g \in B[a,b]$ and $\lambda \ge 0$. A function $u \in C_0[a,b]$ is said to solve  the linear equation of RL type
$(-L_{[a,b]}, \lambda, g,0,0)$ as $(i)$ a  \textit{solution in the domain of the generator}   if $u$ is a solution belonging to $\mathfrak{D}_{[a,b]}^{kill}$; $(ii)$  a  \textit{generalized solution}  if  for all sequence of functions $g_n \in C_0[a,b]$ such that $\sup_n ||g_n|| < \infty$ and $\lim_{n\to \infty} g_n \to g$ a.e.,  it holds that  $u(x) =  \lim_{n\to \infty} w_n(x) $ for all $x \in [a,b]$, where $w_n$ is the unique solution (in the domain of the generator)  to the RL type problem  $(-L_{[a,b]}, \lambda, g_n, 0,0)$.
 \end{definition}

 \begin{definition} For $g \in B[a,b]$ and $\lambda \ge 0$, we say that the equation $(-L_{[a,b]}, \lambda, g,0,0)$ is \textit{well-posed in the generalized sense} if it has a unique generalized solution according to Definition   \ref{Def3Sol}.
 \end{definition}

\begin{theorem} {\rm (\textbf{Well-posedness})} \label{WellP-RL2}
Let $\nu$  be a function defined in terms of two functions $\nu_+$ and $\nu_{-}$  via the equalities in (\ref{nu+-}). Let $\lambda \ge 0$ and  assume that the assumptions of Theorem \ref{L0diffusion-int} hold.   Let $\hat{R}_{\lambda}$ denote the resolvent   operator (or the potential operator if $\lambda = 0$) of the  process $\hat{X}_x$.
\begin{enumerate}
\item [(i)] If $g \in C_0[a,b]$ and  $\left ( \hat{R}_{\lambda}g\right) (x) = 0$ for $x \in \{a,b\}$, then there exists a unique solution in the domain of the generator, $u \in C_0[a,b]$,  to the two-sided RL type equation $(-L_{[a,b]}, \lambda, g, 0,0)$
 given by $u(x) = R_{\lambda}^{[a,b]}g (x)$, where $R_{\lambda}^{[a,b]}$ denotes the resolvent operator (or potential operator if $\lambda = 0$) of the process $X_x^{[a,b]}$.
 \item [(ii)] For any $g \in B[a,b]$, the equation $(-L_{[a,b]},\lambda, g, 0,0)$  has a unique generalized solution $u \in C_0[a,b]$ given by
\begin{equation}\label{StocR-RL2}
u(x) = \mathbf{E} \left [ \int_0^{\tau_{(a,b)} (x)} e^{-\lambda t} g \left( X_{x} (t)\right) dt\right ],
\end{equation}
where $\tau_{(a,b)} (x)$ denotes  the first exit time from the interval $(a,b)$ of the underlying process $X_x$  generated by the operator (\ref{Lfree}).
\item [(iii)] The solution in  (\ref{StocR-RL2}) depends continuously on the function $g$.
\end{enumerate}
\end{theorem}

\begin{proof}
{(i)}\, Theorem \ref{L0diffusion-int} implies that  $(\,-L_{[a,b]*}\,, \hat{\mathfrak{D}}_{*} \, )$ generates a Feller process   $\hat{X}$ and a   strongly continuous semigroup on $C[a,b]$. Then,  the resolvent equation $-L_{[a,b]*}u = \lambda u - g$ has a  unique solution $u \in \hat{\mathfrak{D}}_{*}$    given   by  the resolvent operator $\hat{R}_{\lambda}g$ for $\lambda > 0$ and for any $g \in C[a,b]$ \cite[Theorem 1.1]{dynkin1965}.
 In particular, the latter statement holds for   $g\in C_0[a,b]$ such that $\left ( \hat{R}_{\lambda}g\right) (x) = 0$ for $x \in \{a,b\}$. Further, Theorem  \ref{L0diffusion} implies that  $ \hat{R}_{\lambda}g =  R_{\lambda}^{[a,b]}g $, so that  $ -L_{[a,b]*} u = -L_{[a,b]} u$. Hence, $u$ is a solution to $( -L_{[a,b]}, g, \lambda ,0,0)$ belonging to $\mathfrak{D}_{[a,b]}^{kill}$, as required.

Since  $\tau_{(a,b)}(x):= \inf \{t \ge 0: X_x^{[a,b]}  (t) \notin (a,b)\}$  is the lifetime of the process $X_x^{[a,b]}$,  the definition of $R_{\lambda}^{[a,b]}$ and Fubini's theorem  imply
\begin{equation}\label{uRn} R_{\lambda}^{[a,b]}g(x) =  \mathbf{E}  \left [  \int_0^{\tau_{(a,b)} (x)}  e^{-\lambda t}g\left (X_x^{[a,b]} (t)\right) dt\right ], \end{equation}
yielding (\ref{StocR-RL2}) as the paths of $X_x^{[a,b]}$ and $X_x$ coincide before the time $\tau_{(a,b)}(x)$.
If $\lambda = 0$, then  setting $\lambda= 0$ in (\ref{uRn}) implies (as $\tau_{(a,b)} (x)$ has a finite expectation) that \[||R_0^{[a,b]} g || \le \sup_{x \in [a,b]} \mathbf{E} \left [  \tau_{(a,b)}(x)\right ] < +\infty.\]  Therefore,   the potential operator $R_0^{[a,b]}g$ provides the unique solution for $\lambda=0$ belonging to the domain $\mathfrak{D}_{[a,b]}^{kill}$,  \cite[Theorem 1.1']{dynkin1965}.

\vskip 2pt

$(ii)$\, Take $g \in B[a,b]$ and any sequence   $\{g_n\}$ satisfying Definition \ref{Def3Sol}.   Fubini's theorem and the dominated convergence theorem  applied to (\ref{uRn})  imply   the  convergence of  $\lim_{n \to \infty}  R_{\lambda}^{[a,b]}g_n(x)=:u(x)$, which in turn implies that $u$  is the unique generalized solution to $(-L_{[a,b]}, \lambda, g,0,0)$.

\smallskip

$(iii)$ Follows from the fact that,  for any $\lambda \ge 0$, the equality  (\ref{StocR-RL2}) implies   \begin{equation}\label{boundGgn}
 ||u-u_n||  \le ||g-g_n|| \sup_{x \in [a,b]} \mathbf{E} \left [Ê \tau_{(a,b)}(x) \right ],\end{equation}
 for the solutions $u$ and $u_n$ to equations $(-L_{[a,b]}, \lambda, g, 0,0)$ and $(-L_{[a,b]}, \lambda, g_n,$ $0,0)$, respectively.
\end{proof}

\section{Two-sided equations involving Caputo type operators} \label{sec:6}
\setcounter{section}{6}
\setcounter{equation}{0}\setcounter{theorem}{0}

We now turn our attention to  the well-posedness for the Caputo type equation with general boundary conditions.  We will use that  both operators $-L_{[a,b]*}$ and $-L_{[a,b]}$  coincide on functions $h$ vanishing on $\{a,b\}$.

Suppose  that $u$ solves (\ref{Eq0-C}).  Take any function  $\phi \in \mathfrak{D}_{[a,b]*}^{stop}$ satisfying $\phi(a) = u_a$ and $\phi(b) =u_b$. By Theorem \ref{L0diffusion} we can take, for example, $\phi \in C^2[a,b]$ such that $\phi'\in C_0[a,b]$ with $\left(-L_{[a,b]*} \phi \right) (x) = 0 $ for $x \in \{a,b\}$ and $\phi(a) = u_a$ and $\phi(b) = u_b$.
Define $w(x) := u(x) - \phi (x)$, $x \in [a,b]$,  then
\[ -L_{[a,b]} w(x) = -L_{[a,b]*} w(x)  = -L_{[a,b]*} u(x) + L_{[a,b]*} \phi (x),  \]
as $w$ vanishes at the boundary. Hence,
\begin{align}\nonumber
 -L_{[a,b]} w(x) &= \lambda u(x) - g(x) + L_{[a,b]*} \phi (x), \\
 &= \lambda w(x) + \lambda \phi(x) - g(x) + L_{[a,b]*} \phi(x), \label{E33}\end{align}
yielding the RL type  equation  $(-L_{[a,b]} ,\,\lambda,\, g-L_{[a,b]*}  \phi - \lambda \phi,\,0,\,0)$ for the function $w$. Therefore,   if $w$ is the (possibly generalized) solution to (\ref{E33}), then   $u = w + \phi$ can be considered as  a generalized solution to the  Caputo type equation (\ref{Eq0-C}). This motivates the definition below.

 \begin{definition}(\textbf{Solutions to Caputo type equations})\label{Def3SolC}
 Let $g \in B[a,b]$ and $\lambda \ge 0$.  A function $u \in C[a,b]$ is said to solve  the linear equation (\ref{Eq0-C})  as
 $(i)$ a  \textit{solution in the domain of the generator}   if $u$ is a solution belonging to $\mathfrak{D}_{[a,b]*}^{stop}$; $(ii)$  a  \textit{generalized solution}  if   $u$ can be written as $u= \phi+ w$, where $w$ is the (possibly generalized) solution to the RL type problem  \[(-L_{[a,b]},\, \lambda, \,g-L_{[a,b]*}  \phi - \lambda \phi,\, 0,\,0)\] with $\phi \in C^2[a,b]$ satisfying that  $\phi' \in C_0[a,b]$, $\left( -L_{[a,b]*} \phi\right ) (x) = 0$ in $\{a,b\}$,  $\phi(a) = u_a$ and $\phi(b) = u_b$.
 \end{definition}

 \begin{definition} For  $g \in B[a,b]$ and $\lambda \ge 0$.
We say that the two-sided linear equation (\ref{Eq0-C}) is \textit{well-posed in the generalized sense} if it has a unique generalized solution according to Definition \ref{Def3SolC}.  
\end{definition}

\begin{theorem}
If a generalized solution $u= w + \phi$ exists for the Caputo type linear equation (\ref{Eq0-C}) with $w$ and $\phi$ as in Definition \ref{Def3SolC}, then the solution $u$ is unique and independent of $\phi$.
\end{theorem}

\begin{proof}
Suppose that there are two different solutions $u_j$  for $j\in\{1,2\}$ to equation (\ref{Eq0-C}).  Then,  $u_j = w_j + \phi_j$,  where $w_j$ is the unique solution (possibly generalized) to the RL type equation $(-L_{[a,b]} ,\,\lambda,\, g-L_{[a,b]*} \phi_j - \lambda \phi_j,\,0,\, 0)$ for some  $\phi_j$   satisfying the conditions stated in Definition \ref{Def3SolC}.
Define $u(x) := u_1(x) - u_2 (x)$ for $x \in [a,b]$,  then
\begin{align*}
-L_{[a,b]}  u(x) =-L_{[a,b]*}  u(x) &= -L_{[a,b]*}  u_1(x) + L_{[a,b]*}  u_2(x)
= \lambda u(x).
\end{align*}
Hence, $u$ solves the RL type equation $(-L_{[a,b]} , \lambda, \, g= 0,  0,0)$ whose unique solution (by Theorem  \ref{WellP-RL2}) is $u \equiv 0$, which implies the uniqueness and so the  independence of $\phi$.
\end{proof}

\begin{theorem} \label{WellP-C2} {\rm (\textbf{Well-posedness})}
Let $\lambda \ge 0$. Suppose that the assumptions of Theorem \ref{WellP-RL2} hold.
\begin{enumerate}
\item [(i)] For any $g \in B[a,b]$, the two-sided equation (\ref{Eq0-C})
is well-posed in the generalized sense. The solution admits the stochastic representation
\begin{align}\nonumber
 u&(x) =  u_a\mathbf{E} \left [  e^{-\lambda \tau_{(a,b)} (x)} \mathbf{1}_{\{X_x(\tau_{(a,b)}(x)) \le a  \}} \right ] \\ &+ u_b\mathbf{E} \left [  e^{-\lambda \tau_{(a,b)} (x)} \mathbf{1}_{\{X_x (\tau_{(a,b)} (x)) \ge b  \}} \right ] +    \mathbf{E} \left [   \int_0^{\tau_{(a,b)} (x) } e^{-\lambda t} g \left( X_x (t)\right )dt\right ],\label{StocR-C2}
 \end{align}
 \noindent where  $\tau_{(a,b)}(x)$ and $X_x$  are as in Theorem \ref{WellP-RL2}.
 \item [(ii)] If $g \in C[a,b]$ satisfying that $g(a) = \lambda u_a$, $g(b) = \lambda u_b$ and $ \lambda \hat{R}_{\lambda} g (x) = g(x)$ for $ x \in \{a,b\}$, then  the  solution   (\ref{StocR-C2})  belongs to  $\mathfrak{D}_{[a,b]*}^{stop}$.
\item [(iii)] The solution to (\ref{Eq0-C}) depends continuously on the function $g$ and on the boundary conditions $\{u_a, u_b\}$.
 \end{enumerate}
 \end{theorem}

 \begin{proof}
   $(i)$\, Theorem  \ref{L0diffusion-int} implies that  the operator $(\,-L_{[a,b]*}\,,\hat{ \mathfrak{D}}_{*} \, )$ generates a Feller   process   $\hat{X}$ on $[a,b]$ and also ensures that $\tau_{(a,b)}(x)$ has a finite  expectation. Let us take any function $\phi \in C^2[a,b]$ satisfying the conditions from  Definition \ref{Def3SolC}.  Then  (by Theorem \ref{WellP-RL2})
  the generalized solution $w$ to the RL type equation $(-L_{[a,b]}, g - \lambda \phi - L_{[a,b]*} \phi, \lambda, 0,0)$ is given by $w = I - II$, where
 \begin{align*}
I &:= \mathbf{E} \left [ \int_0^{\tau_{(a,b)} (x)} e^{-\lambda t}   g\left( X_{x}^{[a,b]} (t)\right)dt\right ]  \\
II &:=  \mathbf{E} \left [ \int_0^{\tau_{(a,b)} (x) } e^{-\lambda t}    ( \lambda  + L_{[a,b]*}) \phi  \left( X_x^{[a,b]} (t)\right)    dt\right ].
 \end{align*}
Thus, $u = w + \phi$ is (by definition) the generalized solution to (\ref{Eq0-C}).  Using the  martingale
 \[Y(r):= e^{-\lambda r} \phi \left (   X_{x}^{[a,b]*} (r)\right) + \int_0^r e^{-\lambda s} (\lambda + L_{[a,b]*}) \phi \left (  X_{x}^{[a,b]*} (s)\right ) ds\]
 and the stopping time $\tau_{(a,b)}(x)$, Doob's stopping theorem yields
 \[ II = \phi(x) - \mathbf{E} \left [Ê  e^{-\lambda \tau_{(a,b)} (x) }  \phi \left (X_{x}^{[a,b]*} \left ( \tau_{(a,b)} (x) \right ) \right )\right ]\]
 which in turn implies
\begin{align}\nonumber
 u(x) =  \mathbf{E} &\left [  e^{-\lambda \tau_{(a,b)} (x)}  u \left(  X_x^{[a,b]*} \left( \tau_{(a,b)} (x) \right) \right ) \right ] \,\, + \\ &\quad\quad\quad
 + \quad   \mathbf{E} \left [   \int_0^{\tau_{(a,b)} (x) }\! e^{-\lambda t} g \left( X_x^{[a,b]*} (t)\right )dt\right ],\label{Ctype}
 \end{align} 
as   $\phi \left (X_{x}^{[a,b]*} \left ( \tau_{(a,b)} (x) \right ) \right ) = u\left ( X_{x}^{[a,b]*} \left( \tau_{(a,b)} (x)\right)\right )$ by assumption. Finally, since at the random time $\tau_{(a,b)} (x)$ the process $X_{x}^{[a,b]*}$ takes  either the value $a$ or the value $b$,   the first term in the r.h.s of (\ref{Ctype})  can be written as
 \begin{align}\nonumber
\mathbf{E} &\left [  e^{-\lambda \tau_{(a,b)} (x)}  u \left(  X_x^{[a,b]*} \left( \tau_{(a,b)} (x) \right) \right ) \right ] \\
=\, &u_a\mathbf{E} \left [  e^{-\lambda \tau_{(a,b)} (x)} \mathbf{1}_{\{X_x(\tau_{(a,b)}(x)) \le a  \}} \right ] + u_b\mathbf{E} \left [  e^{-\lambda \tau_{(a,b)} (x)} \mathbf{1}_{\{X_x (\tau_{(a,b)} (x)) \ge b  \}} \right ], \nonumber
 \end{align}
where $X_x$ is the underlying  process (see (\ref{Lfree})), which yields the result (\ref{StocR-C2}).
$(ii)$ Take $g \in C[a,b]$ such that $\lambda \hat{R}_{\lambda} g (x) = g(x)$ for $x \in \{a,b\}$. Item (i) above ensures that the solution is given by $u = w + \phi$, where $w$ is a RL type solution  and $\phi$ is a function satisfying the conditions given in Definition \ref{Def3SolC}.  By  Theorem \ref{WellP-RL2},  $w$ belongs to $\mathfrak{D}_{[a,b]}^{kill}$ whenever
\[g(a) = \lambda u_a + (-L_{[a,b]*} \phi)(a)\quad \text{and} \quad  g(b) = \lambda u_b + (-L_{[a,b]*} \phi)(b).\] But, by Theorem \ref{L0diffusion}, $(-L_{[a,b]*} \phi)(a) = (-L_{[a,b]*} \phi)(b) = 0$ because $\phi \in \mathfrak{D}_{[a,b]*}^{stop}$.  Further, assumption $\lambda \hat{R}_{\lambda} g (x) = g(x)$ in $\{a,b\}$ implies $-L_{[a,b]*} u (x) = 0$ for $x \in \{a,b\}$, which in turn implies $-L_{[a,b]*}u =  -L_{stop} u$. Hence,  Theorem \ref{L0diffusion} guarantees that  $u \in \mathfrak{D}_{[a,b]*}^{stop}$ whenever $g(a)= \lambda u_a$ and $g(b) = \lambda u_b$, as required.

$(iii)$ Follows  from the representation (\ref{StocR-C2})  and  from (\ref{boundGgn}).
  \end{proof} 

\textbf{Case $-A$ vanishing or $-A = \gamma(\cdot)\frac{d}{dx}$}.  
For these cases, an additional assumption is needed to guarantee the regularity in expectation of the boundary points $\{a,b\}$.
\begin{enumerate}
\item [\textbf{(H1)}] There exist a constant $C > 0$ and $q \in (0,1)$ such that
\begin{align}\nonumber
 \int_{-\infty}^0 \min (|y|, \epsilon) \nu(a,y)dy > C \epsilon^{q} \quad\text{ and } \quad
 \int_{0}^{\infty} \min (y, \epsilon) \nu(b,y)dy > C \epsilon^{q}.
\end{align}
\end{enumerate}

\begin{theorem}\label{L0jump}
 Let  $\lambda \ge 0$.  Assume that the function $\nu$ associated with $\nu_+$ and $\nu_{-}$ (defined via the equalities in (\ref{nu+-}))  satisfies  assumptions (H0) and (H1). Then,  Theorem \ref{WellP-RL2} and Theorem  \ref{WellP-C2} also hold with $\alpha \equiv 0$ and with either $\gamma \equiv 0$ or  $\gamma \in C_0^1[a,b]$.
\end{theorem}

\begin{proof}
Since the reasoning is  same as before, we omit the details.
\end{proof}

To finish this section, let us consider the following result related to the exit time of Feller processes  from bounded  intervals and generalized fractional equations of Caputo type.
Let $X_x$ be the process generated by (\ref{Lfree}). Define $\Pi_a(x)$ and $\Pi_b (x)$ as the event  that the process $X_x$ leaves the interval $(a,b)$ through the lower boundary $a$, and  through the upper boundary $b$, respectively, i.e.
\begin{equation}\nonumber
\Pi_a (x) := \left \{  X_{x}\left (\tau_{(a,b)} (x) \right ) \le a \right \}\quad \text{ and } \quad \Pi_b (x) := \left \{  X_x\left (\tau_{(a,b)} (x) \right ) \ge b \right \}.
 \end{equation}
Let $H^D(x,\cdot)$ be the \textit{potential measure} for the process $X_x$ (see, e.g. \cite{Getoor}) defined by
\[  H^D(x,dy) := \mathbf{E} \left [  \int_0^{\infty} \mathbf{1}_{\{ X_x(t) \in dy\}} \mathbf{1}_{\{ \forall s \le t, X_x (s) \in D\}} dt \right ]. \]

\begin{corollary}\label{twosidedT}
Under the assumptions of Theorem \ref{WellP-C2}, the generalized solution to the two-sided equation  (\ref{Eq0-C}) with $\lambda= 0$  is given by
\begin{equation}\label{PaPbH}
 u(x) =  u_a \mathbf{P} [\Pi_a(x)] + u_b \mathbf{P}[\Pi_b (x)] +  \int_a^b g(y) H^{(a,b)}(x,dy).
\end{equation}
In particular, $u(x) = \mathbf{E} \left [ \tau_{(a,b)} (x)\right ]$   is the  generalized solution to the two-sided equation with $g = -1$ and $u_a = u_b = 0$. Further, $u(x) = \mathbf{P} [\Pi_a (x)]$ is the generalized solution to the equation with  $g = 0$, $u_a = 1$ and $u_b = 0$, whereas $u(x) = \mathbf{P} [\Pi_b(x)]$ solves the equation with $g = 0, u_a = 0$ and $u_b = 1$.
\end{corollary}

\section{Examples} \label{sec:7}
\setcounter{section}{7}
\setcounter{equation}{0}\setcounter{theorem}{0}

\begin{example}  Consider the two-sided  Caputo fractional equation
\begin{align}\nonumber
D_{-1+*}^{\beta} w(x) &+ D_{+1-*}^{\beta} w(x) = - \lambda w(x) + g(x), \quad x \in (-1,1) \\ \label{S-Wat}
&w(-1) = 0 = w(1).
\end{align}
By Theorem \ref{L0jump}, the boundary value problem  (\ref{S-Wat}) is well-posed in the generalized sense for any $g \in B[-1,1]$ with solution  \begin{align}
 w(x) &=    \mathbf{E} \left [   \int_0^{\tau_{(-1,1)} (x) } e^{-\lambda t} g \left( X_x^{\beta} (t)\right )dt\right ], \quad \lambda \ge 0,\nonumber
 \end{align}
 where  $X_{x}^{\beta}$ is a symmetric stable process with exponent $\beta \in (0,1)$  and
 \[ \tau_{(-1,1)} (x) := \inf \left \{ t \ge 0\,:\, X_x^{\beta} (t) \notin (-1,1) \right\}.\]
Further,  if $g = 1$ and $\lambda = 0$,  then the \textit{ mean exit time} $\mathbf{E} \left [ \tau_{(-1,1)} (x)\right ]$ is the unique generalized solution to  (\ref{S-Wat}). Moreover, by Theorem 2.1 in \cite{watanabe}, we obtain the explicit solution
\[ w(x)= \frac{(1-x^2)^{\beta/2}}{\Gamma (\beta + 1)}.\]
\end{example}

\begin{example} Consider now the two-sided Caputo fractional equation:
\begin{align}\nonumber
D_{-1+*}^{\beta} h(x) &+ D_{+1-*}^{\beta} h(x) = 0, \quad x \in (-1,1)\quad \beta \in(0,1), \\
&h(-1) = 0, \quad  h(1) = 1.\label{S-WatC}
\end{align}
Corollary \ref{twosidedT} gives the unique generalized solution \[h(x) = \mathbf{P} \left [  X_x^{\beta} (\tau_{(-1,1)} (x)) \in [1, \infty) \right ],\] which is given explicitly by  \cite[Formula 3.2]{watanabe}
\begin{equation}\label{Pa}
  h(x) = 2^{1-\beta} \frac{\Gamma(\beta)}{\Gamma(\beta /2)^2} \int_{-1}^x (1-y^2)^{\frac{\beta}{2} -1}dy.
  \end{equation}
Furthermore, again by Corollary \ref{twosidedT}, the equation
\begin{align}\nonumber
D_{-1+*}^{\beta} v(x) &+ D_{+1-*}^{\beta} v(x) = 0, \quad x \in (-1,1),\quad\beta \in(0,1),  \\ \label{S-WatCC}
&v(-1) = 1, \quad  v(1) = 0.
\end{align}
has solution  $v(x) = 1 - h(x).$
\end{example}

\begin{example} The two-sided Caputo fractional equation
\begin{align}\nonumber
 D_{-1+*}^{\beta} u(x) &+ D_{1-*}^{\beta} u(x) = g(x)\quad x \in (-1,1),\quad \beta \in (0,1), \\
&u(-1) = u_{-1}, \quad  u(1) = u_{1}, \quad \quad u_{-1}, u_1 \in \mathbb{R},\label{GenC}
\end{align}
has a unique generalized solution (Corollary \ref{twosidedT}) which rewrites
\[  u(x) = (u_{1} -u_{-1}) h(x) + u_{-1} + \int_{-1}^1 g(y) H_{\beta}^{(-1,1)}(x,y)dy,\]
where $h(x)$ is the function given in (\ref{Pa}), and $H_{\beta}^{(-1,1)}(x,y)$ (the density of the potential measure of the process $X_x^{\beta}$) is given by \cite{watanabe}
\[  H_{\beta}^{(-1,1)}(x,y) =  2^{-\beta} \pi_{-1/2} \frac{\Gamma (1/2)}{(\Gamma (\beta/2))^2} \int_0^z (r+1)^{-\frac{1}{2}}r^{\frac{\beta}{2} -1} |x-y|^{\beta-1}dr,\]
with $ z = (1-x^2) (1-y^2)/ (x-y)^2.$
\end{example}

\begin{remark}
Observe that all the explicit solutions $w,v, h$ and $u$ above are smooth solutions since they belong to $ C[-1,1] \cap C^1(-1,1)$.
\end{remark}
\section{Proofs}  \label{sec:8}
\setcounter{section}{8}
\setcounter{equation}{0}\setcounter{theorem}{0}

Firstly, let us observe that
for $f \in C^1[a,b]$, by setting $g(x) = f'(x)$ we can rewrite
\begin{align}\label{dM*}
-L_{[a,b]*}^{(\nu)} f(x) &= M_{*}^{(\nu)} g(x) := \int_{a-x}^{b-x} \int_{x}^{x+y} g(z)dz  \nu(x,y)dy +  \\ 
&+ \int_x^b g(z)dz \int_{b-x}^{\infty} \nu(x,y)dy + \int_x^a  g(z)dz  \int_{-\infty}^{a-x} \nu(x,y)dy. \nonumber
\end{align}

 \subsection{Proof of Theorem \ref{L0diffusion-int}}

\begin{proof}
$(i)$ Let us approximate $-L_{[a,b]*}$ by a family of   operators \hfill \break 
 $\left (\,-L_{h*} \,\right)_{h\in(0,1]}$ defined by
   \begin{equation}\label{Lh*}
    -L_{h*}:= -L_{[a,b]*}^{(\nu_h)} - A^{(\gamma,\alpha)},
   \end{equation}
   where   $\nu_h (x,y) := \Phi_h(x,y) \nu(x,y)$ with $\Phi_h (x,y)$ being a smooth function on $[a,b] \times \mathbb{R}$, which equals 1 on the set $
  \{|y| > h,x \in [a+h,b-h]\}$  and vanishes near the boundary; and the operator   $(-A^{(\gamma,\alpha)}, \mathfrak{D}_A)$ is  the generator of a  diffusion on $[a,b]$  with reflecting boundaries \{a,b\} (see, e.g. \cite[Chapter V, Section 6]{batach})
with a domain \[\mathfrak{D}_A := \left \{ f \in C[a,b]\,:\, -A^{(\gamma,\alpha)} f \in C [a,b],\, f'(a)=0, \,\, f'(b) = 0 \right \}.\]
Then, for each $h \in (0,1]$ the operator $-L_{h*}$  decomposes as   a  diffusion on $[a,b]$   perturbed by the bounded operator $-L_{[a,b]*}^{(\nu_h)}$ on $C[a,b]$, so that by  perturbation theory (see, e.g., \cite[Theorem 1.9.2]{KV0})   the operator    $(-L_{h*}, \mathfrak{D}_A)$  generates a Feller  semigroup   $T_t^h$ on $C[a,b]$.  This semigroup is the unique (bounded) solution to the evolution equation
\begin{equation}\label{evLh}
\frac{d}{dt} f_t (x) = - L_{h*} f_t(x), \quad f_0 = f \in \mathfrak{D}_A.
\end{equation}
Moreover,   due to the smoothness  assumptions on   $\gamma, \alpha$ and $\nu$, the spaces  $ \{ f\in C^j[a,b]: f' \in C_0[a,b] \}$ for $ j\in \{2,3\}$ are invariant cores of $-L_{h*}$ \cite[Theorem 1.9.2,(iii)]{KV0}.  Hence, if $ f \in C^3[a,b]$ with $f' \in C_0[a,b]$,  then  $T_t^h f \in C^3[a,b]$  and  $-L_{h*} T_t^h f \in C^1[a,b]$.

  Differentiating (\ref{evLh}) with respect to  $x$, rearranging terms and using (\ref{dM*}),    yield the evolution equation for $g_t(x) = f'_t (x)$  given by
\begin{equation}\label{evLg}
\frac{d}{dt} g_t(x)
=\,\, -\bm{L}^{h,(1)} g_t(x),
\end{equation}
where
\begin{align}\label{dLtilde}
 -\bm{L}^{h,(1)} g(x):&=  - A^{(\gamma + \alpha',\alpha)} g(x) + \left [ - L_{[a,b]}^{(\nu_h)}   - M_{*}^{(\partial_x\nu_h)}  + \gamma'(x) \right ] g(x).
  \end{align}
 Since (by assumption) $\alpha'$ vanishes  on $\{a,b\}$,  the operator $-\bm{L}^{h,(1)}$ decomposes  as a diffusion  $- A^{(\gamma+\alpha', \alpha)}$ on $[a,b]$ (with reflecting boundaries) perturbed by  the bounded operator  $K_h$ on $C[a,b]$ given by
 \[  K_h := - L_{[a,b]}^{(\nu_h)}   - M_{*}^{(\partial_x\nu_h)}  + \gamma'(\cdot). \]
  Hence, $-\bm{L}^{h,(1)}$     generates a strongly continuous semigroup of contractions on $C[a,b]$,    denoted by  $T_t^{h,(1)}$.
Due to the invariance of the space $\{f \in C^3[a,b]\,:\, f' \in C_0[a,b]\}$,  it follows that $ \frac{d}{dx}(T_t^h f )(x) = \left(T_t^{h,(1)} f'\right)(x)$ for  $f$ in the latter space.
Now,  the perturbation series representation for the semigroup $T_t^{h,(1)}$ \cite[Equality 1.78, p. 52]{KV0}) implies
\begin{align}\label{seriesD1}
|| T_t^{h,(1)} f'|| \le ||  f' || + \sum_{m=1}^{\infty} \frac{(t\,||K_h||)^m }{m!} || f'||.
\end{align}
 Thus, as $K_h$ is uniformly bounded in $h$ due to  the bounds   from assumption (H0),
 the derivative  $\frac{d}{dx}\left(T_t^h f\right)(x) $ is uniformly bounded  in $h$ and  $t \le t_0$  whenever $f \in C^3[a,b]$ with $f' \in C_0[a,b]$.

Let us now  write (see \cite[Lemma 19.26, p. 385]{Kall})
\begin{align*}
(T_t^{h_1} - T_t^{h_2}) f = \int_0^t T_{t-s}^{h_2} \left (-L_{h_1*}+ L_{h_2*}\right ) T_{s}^{h_1} f\,ds,
\end{align*}
for $0<h_2 \le  h_1 < 1$ and $f \in C^3[a,b]$ with $f'\in C_0[a,b]$.   Since $T_t^{h_1} f$ is differentiable (with derivative uniformly bounded in $h$ given by $T_t^{h_1,(1)}f'$),  we can  estimate (by mean value theorem)
\begin{align*}
\Big | \left (-L_{h_1*}+ L_{h_2*}\right ) T_s^{h_1} f(x) \Big | \, &\le \, \int_{h_2 \le |y| \le h_1} \Big | T_s^{h_1} f(x+y) - T_s^{h_1} f(x)   \Big | \nu(x,y)dy \\
&\le \int_{h_2 \le |y| \le h_1}  || T_s^{h_1,(1)} f'|| |y|  \nu(x,y)dy \\
&=  o(1) ||T_s^{h_1,(1)} f'|| = o(1)||f||_{C^1}, \quad h_1 \to 0.
\end{align*}
The last equality holds due to the assumption (H0)  (i.e, the uniform bound of the first moment of $\nu$ and its tightness property).
Therefore,
\begin{equation}\label{Shf-C}
|| \left(T_t^{h_1} - T_t^{h_2} \right)f|| = o(1) t ||f||_{C^1}.
\end{equation}
 Thus,  for each $f \in C^3[a,b]$ with $f'\in C_0[a,b]$,  the family $\{T_t^h f\}$ converges  to a limiting  family $\{T_t f\}$ as $h \to 0$.
 It follows then that  the limiting family forms a strongly continuous   semigroup of contractions  on $C[a,b]$ (by standard approximation arguments). Now write
\[ \frac{T_t f -f}{t} = \frac{T_t f - T_t^h f}{t} + \frac{T_t^h f - f}{t}. \]
Using the estimate (\ref{Shf-C}), we conclude that $\{ f \in C^3[a,b]\,:\,f'\in C_0[a,b]\}$   belongs to the  domain of the generator,  and that the generator  is given by $-L_{[a,b]*}$ as
\[ 
\lim_{t\downarrow 0} \frac{T_t f -f}{t} = \lim_{h\downarrow 0}\lim_{t\downarrow 0}\frac{T_t f - T_t^h f}{t} + \frac{T_t^h f - f}{t} = -L_{[a,b]*} f. \]
Now, take $f \in C^2[a,b]$ and  $\{f_n\}  \subset \{ f \in C^3[a,b]\,:\, f' \in C_0[a,b]\}$ such that $f_n \to f$ uniformly as $n\to \infty$.   Since the operator $-L_{[a,b]*}$ is closed \cite[Corollary 1.6]{EK} and  $-L_{[a,b]*} f_n \to  g$ as $n \to \infty$ for some $g$, it follows that $g = -L_{[a,b]*} f$ and $f \in \hat{\mathfrak{D}}_{*}$. Therefore,  the space $\{f \in C^2[a,b]\,:\, f' \in C_0[a,b]\}$ also belongs to the domain of the generator, as required.

$(ii)$ Take the function $f_w(x) = (x-a)^w$ for some sufficiently small $w \in (0,1)$.    We will prove that  $\left(-L_{[a,b]*} f_w \right)(x) < 0$ for $x \in (a,c)$ and  $c \in (a,b)$ (see method of Lyapunov functions, e.g., \cite[Proposition 6.3.2]{KV0}). Since
 \begin{equation}\nonumber \left( -L_{[a,b]*} f_w\right) (x) = -L_{[a,b]*}^{(\nu)}f_w (x) + w \gamma(x) (x-a)^{w-1} + w (w-1)\alpha(x) (x-a)^{w-2},
 \end{equation}
  when $\gamma(a) = 0$ and $\alpha(a) > 0$, then $\left( -L_{[a,b]*} f_w\right) (x)  < 0$ as  the first two terms in the r.h.s of the previous equality  are dominated by the last term which tends to $-\infty$ as $x \to a$.  The regularity for $x=b$ is proved analogously but with  $f_w(x) = (b-x)^{w}$. Finally, Proposition 6.3.2 in \cite{KV0}   implies   the finite expectation of  $\hat{\tau}_{(a,b)}(x)$. 
  \end{proof}   

 \subsection{Proof of Theorem \ref{L0diffusion}}

 \begin{proof}  
 $(i)$ Theorem \ref{L0diffusion-int} implies that  $(-L_{[a,b]*}, \hat{\mathfrak{D}}_{*})$ generates a  Feller process $\hat{X}_x$  on $[a,b]$ and ensures the regularity in expectation of the boundary points  $\{a,b\}$. Hence, the stopped process  $X_x^{[a,b]*} := \{  \hat{X}_x (s \wedge \tau_{(a,b)}(x))\}_{s \ge 0}$   is also a Feller process on $[a,b]$ \cite[Theorem 6.2.1, Chapter 6]{KV0}. Let us denote by $(-L_{stop},\,\mathfrak{D}_{[a,b]*}^{stop})$ the generator of the stopped process  with a domain denoted by $\mathfrak{D}_{[a,b]*}^{stop}$. By definition of $X_x^{[a,b]*}$  the states $\{a,b\}$ are absorbing,  which implies that  $(-L_{stop}f )(x) = 0$ for $x \in \{a,b\}$ and  $f \in D_{[a,b]*}^{stop}$.

Take now $f \in \hat{\mathfrak{D}}_*$ such that $-L_{[a,b]*} f (x) = 0 $ in $\{a,b\}$. Since the domain of the generator is given by the image of its  resolvent operator (say $\hat{R}_{\lambda}$),  given $f \in \hat{ \mathfrak{D}}_*$  there exists $g \in C[a,b]$ such that $f = \hat{R}_{\lambda} g$.

Using that $f$ solves the resolvent equation
\[  \lambda \hat{R}_{\lambda }g + L_{[a,b]*} f = g, \]
and that (by assumption)  $-L_{[a,b]*} f (x) = 0$ for $xÊ\in \{a,b\}$, we get
\begin{equation}\label{boundaries}
f(a) = \hat{R}_{\lambda} g(a) = g(a) /\lambda \quad \text{ and }\quad  f(b) =\hat{R}_{\lambda} g(b) = g(b) /\lambda.\end{equation}
Moreover, Dynkin's formula implies
\begin{align*}
 \hat{R}_{\lambda} g(x) &= \mathbf{E}\left [ \int_0^{\tau_{(a,b)}(x)} e^{-\lambda s} g \left ( \hat{X}_x (s)\right )ds\right ] + \mathbf{E} \left [e^{-\lambda \tau_{(a,b)}(x)}  f \left( \hat{X}_x(\tau_{(a,b)} (x) )\right)\right ]
\end{align*}
for each $x \in (a,b)$.
Using that the paths of the process $\hat{X}_x$ and $X_x^{[a,b]*}$ coincide before the first exit time $\tau_{(a,b)}(x)$, the previous expression becomes
\begin{align*}
\hat{R}_{\lambda }g(x) &= \mathbf{E}\left [ \int_0^{\tau_{(a,b)}(x)} e^{-\lambda s} g \left ( X^{[a,b]*}_x (s)\right )ds\right ] \,\, + \\ & \quad\quad  + \,\,\,\mathbf{E} \left [  e^{-\lambda \tau_{(a,b)}(x)}\,  \left (  f(a)\mathbf{1}_{\{\tau_a < \tau_b\}} + f(b) \mathbf{1}_{\{\tau_b < \tau_a \}} \right )\,\right ],
\end{align*}
 where $\tau_a$ and $\tau_b$ denote the first exit time through the boundary point $a$ and $b$, respectively.  Finally, plugging  (\ref{boundaries}) into the second term of the r.h.s of the last formula we get that $f = \hat{R}_{\lambda} g   = R_{\lambda}^{[a,b]*} g$, where $R_{\lambda}^{[a,b]*}$ denotes the resolvent operator of  $X^{[a,b]*}$.  Therefore, $f \in \mathfrak{D}_{[a,b]*}^{stop}$ as  there exits $g \in C[a,b]$ such that $f = R_{\lambda}^{[a,b]*}g$, which in turn implies that $-L_{stop} f = -L_{[a,b]*} f$. \\
 $(ii)$ Follows the same arguments as before, so that we omit the details.
  \end{proof}  

\section*{Acknowledgements}

The first author is supported by Chancellor International Scholarship and the Department of Statistics through the University of Warwick, UK.




 \bigskip \smallskip

 \it

 \noindent
Department of Statistics \\
University of Warwick, Coventry CV4 7AL, UK \\ [6pt]
  $^1$ e-mail: M.E.Hernandez-Hernandez@warwick.ac.uk\\
  $^2$  e-mail: V.Kolokoltsov@warwick.ac.uk
\end{document}